\documentclass[11pt,a4paper]{article}
\date{}
\usepackage{enumerate,amsmath,amsthm,amssymb,latexsym,
amsfonts,amscd}
\usepackage[all]{xy}

\title{Rational and iterated maps, degeneracy loci, \\and the generalized Riemann-Hurwitz formula}

\author{James  F. Glazebrook
and Alberto Verjovsky \thanks{This work was partially supported by PAPIIT (Universidad
Nacional Aut\'onoma de M\'exico) \#IN103914.}}

\theoremstyle{plain}

\newtheorem{proposition}{Proposition}[section]
\newtheorem{theorem}{Theorem}[section]
\newtheorem{corollary}{Corollary}[section]

\theoremstyle{definition}

\newtheorem{example}{Example}[section]
\newtheorem{remark}{Remark}[section]

\numberwithin{equation}{section}

\newcommand{\Aut}{{\rm Aut}}

\newcommand{\codim}{{\rm codim}}
\newcommand{\coker}{{\rm coker}}

\newcommand{\Fix}{{\rm Fix}}

\newcommand{\Gr}{{\rm Gr}}

\newcommand{\Hom}{{\rm Hom}}
\newcommand{\Ker}{{\rm Ker}}

\newcommand{\IM}{{\rm Im}}

\newcommand{\rank}{{\rm rank}}

\newcommand{\Res}{{\rm Res}}

\newcommand{\SO}{{\rm SO}}
\newcommand{\SU}{{\rm SU}}
\newcommand{\Sp}{{\rm Sp}}

\newcommand{\Sing}{{\rm Sing}}

\renewcommand{\a}{\alpha}
\newcommand{\be}{\beta}

\newcommand{\K}{\mathcal K}

\newcommand{\U}{\mathcal U}

\newcommand{\bC}{\mathbb{C}}

\newcommand{\bN}{\mathbb{N}}

\newcommand{\bR}{\mathbb{R}}

\newcommand{\bZ}{\mathbb{Z}}

\newcommand{\fX}{\mathfrak{X}}

\newcommand{\lra}{\longrightarrow}

\newcommand{\ovsetl}[1]{\overset {#1}{\lra}}

\newcommand{\thra}{\twoheadrightarrow}

\newcommand{\what}{\widehat}

\newcommand{\sfH}{\mathsf{H}}

\newcommand{\del}{\partial}

\newcommand{\med}{\medbreak}

\newcommand{\medn}{\medbreak \noindent}

\begin{document}

\maketitle
\begin{center}\emph{Dedicated to Jos\'e Seade on his 60th birthday}
\end{center}

\begin{abstract}
We consider a generalized Riemann-Hurwitz formula as it may be applied to rational maps between
projective varieties having an indeterminacy set and fold-like singularities. The case of a
holomorphic branched covering map is recalled. Then we see how the formula can be applied to iterated
maps having branch-like singularities, degree lowering curves, and holomorphic maps having a fixed point set.
Separately, we consider a further application involving the Chern classes
of determinantal varieties when the latter are realized as the degeneracy loci of certain vector bundle morphisms.
\end{abstract}

\medbreak

\med



\section{Introduction}\label{introduction}

The two-fold aim of this paper is firstly to consider (generalized) higher dimensional versions of the classical Riemann-Hurwitz formula
as initially applied to rational maps of complex projective varieties
$f:X \lra Y$, where $X$ and $Y$ have the same complex dimension. The main
results presented here (Theorem \ref{main-1} and Theorem
\ref{main-2}) are derived from the general setting of \cite{GGV}
formulated mainly in
the category of CW-complexes, and then applying the basics of the topological theory of characteristic classes of $G$-bundles (for suitable groups $G$) and associated vector bundles.
Other versions of a generalized
Riemann-Hurwitz formula, such as in the differentiable category, had previously
been obtained in \cite{Brasselet1,Schwartz1,Vanque}. In this first part we
will be applying the main results to operations involving rational
maps, as for instance, realized in various algebraic-geometric and complex-dynamical constructions (beyond, that is, the familiar case of holomorphic
branched covering maps). Applications and examples include:
\begin{itemize}
\item[(1)] Iterates of rational maps $f: X \lra Y$ \cite{Bonifant1}(cf. \cite{Fornaess1}).

\item[(2)] Rational self-maps $f: X \lra X$,  with respect to their
  fixed point sets \cite{Abate1}, and degree lowering curves \cite{Fornaess1}.
\end{itemize}
In the second part, we will apply those same main results and the allied constructions to the study of
determinantal varieties when the latter are realized as the degeneracy loci of morphisms
$\psi: E \lra F$,
of complex vector bundles over $X$. From a general formula established in
\S\ref{degeneracy-locus}, we pay attention to two particular cases:
general symmetric bundle maps \cite{HT} and flagged bundles \cite{Fulton1}.

Throughout, the formulas are established in terms of Chern classes, as by now the traditional method
for studying invariants in the algebraic-geometric category. Specifically, such formulas
regulate the necessary topological
conditions for the existence of a given class of rational maps (or morphisms), just as the
classical Riemann-Hurwitz formula applies when studying holomorphic maps of
algebraic curves (viz. compact Riemann surfaces).

The construction and main results of \cite{GGV} were adapted in \cite{GV1} to cover the case of generalized monoidal transformations.
A further work \cite{GV2} is intended to further elaborate on this construction, as well as to bring into focus a number of results obtained
by other authors concerning the blowing-up process of (singular) Chern classes following the original study undertaken by I. R. Porteous \cite{Porteous1,Porteous2}
and R. Thom \cite{Th} (see also
\cite{Aluffi,Gitler1}).


\section{The topological background}\label{topological}

Given a topological group $G$, we start by recalling from \cite{GGV} a general result valid in the characteristic ring of $G$-bundles when defined initially within the category of CW-complexes.

\subsection{Adapted pairs}\label{adapted}

Following \cite{GGV}, let $\Lambda$ be a given (commutative) coefficient ring and $(M,M_1)$ a pair of CW-complexes, with $M$ of dimension $n$ and $M_1$ a subcomplex of codimension $r \geq 2$, so that $H_q(M, \Lambda) = 0$ for $q > n$, and $H_q(M_1, \Lambda) = 0$ for $q \geq n-1$. Then the pair $(M, M_1)$ is called $(n, \Lambda)$-\emph{adapted} if $H_n(M, \Lambda) \cong \Lambda$, and the following condition holds:
there exists a neighborhood $N(M_1)$ of $M_1$, such that $M_1$ is a deformation retract of the interior $N^{0}(M_1)$ of $N(M_1)$, and the inclusion map $p: M \lra (M, M-N^{0}(M_1))$ induces an isomorphism
\begin{equation}
    p_{*}: H_n(M) \ovsetl{\cong} H_n(M, M-N^{0}(M_1)).
\end{equation}

Having defined an adapted pair we next form the subspace $\K$ of $M \times I$, where
\begin{equation}
\K = (M \times \del I) \cup ((M, M-N^{0}(M_1)) \times I),
\end{equation}
together with the double $S(M_1) \subset \K$, given by
\begin{equation}
S(M_1) = (\del N(M_1) \times I) \cup (N(M_1) \times \del I).
\end{equation}
Now let
$$
\K_1 = (M \times \{0\}) \cup ((M, M-N^{0}(M_1)) \times [0, \frac{3}{4}],
$$
and
$$
\K_2 = (M \times \{1\}) \cup ((M, M-N^{0}(M_1)) \times [0, \frac{1}{4}],
$$
and let $S(M_1)_i = S(M_1) \cap \K_i$, for $i=1,2$.

By this construction, the spaces $\K_i$ are homotopically equivalent to $M_1$ and the spaces $\K/\K_1$ and $S(M_1)/S(M_1)_1$, are
both homotopically equivalent to the generalized Thom space $M /(M - N^{0}(M_1))$.

It follows from the cofibration
\begin{equation}
S(M_1)_1 \lra S(M_1) \lra S(M_1)/S(M_1)_1
\end{equation}
and the above definition that
\begin{equation}
H_n(S(M_1)_1, \Lambda) \cong  H_n(S(M_1)/S(M_1)_1) \cong \Lambda.
\end{equation}
A choice of generators for $H_n(M, \Lambda)$ and $H_n(S(M_1), \Lambda)$ will be called an orientation or a fundamental class $[M]$ of $M$ and $[S(M_1)]$ of $S(M_1)$, respectively.

\begin{remark}\label{top-example}
Observe that the conditions defining an $(n, \Lambda)$-adapted pair above are immediately satisfied when $M$ is a closed (compact without boundary) connected orientable $n$-manifold, and $M_1$ is a closed connected and orientable submanifold of codimension $r \geq 2$, with $\Lambda$
 any coefficient ring. This example also applies to topological, PL, as well as smooth (sub)manifolds, with $S(M_1)$ a corresponding normal sphere bundle
 (also closed, connected and orientable for $H_n(S(M_1), \Lambda)  \cong \Lambda$, given the above topological type of $M_1$).

More generally, $M$ could be considered as an
orientable simple $n$-circuit (or `triangulated pseudomanifold' in the sense of \cite{Goresky}; see also \cite{ES,Lefschetz1}), and $M_1$ an arbitrary subcomplex of codimension $r \geq 2$, such that $(M, M_1)$ satisfies the conditions of an $(n, \Lambda)$-adapted pair.
\end{remark}

Let $BG$ denote the classifying space of the topological group $G$. If $P \in H^q(BG, \Lambda)$ is a cohomology class, and $E \lra X$ is a $G$-bundle over a space $X$, then we shall denote by $P(E) \in H^q(X, \Lambda)$ the class defined by the characteristic class $P(E) = \Phi^*_E(P)$ where $\Phi_E: X \lra BG$ is the classifying map (for the basic details see e.g. \cite{BT,Chern,Hirzebruch}).

If $\tau \in H_q(X, \Lambda)$, then we denote the Kronecker pairing by $\langle P(E), \tau \rangle$, which by definition is zero in the case where the cohomological degree of $P(E)$ is different from the homological degree (or dimension) of the cycle $\tau$.

We have then the general result
from \cite[Theorem 1.1]{GGV}:
\begin{theorem}\label{main-1}
Suppose $(M,M_1)$ is an $(n, \Lambda)$-adapted pair and $E, F$ are $G$-bundles over $M$, such that on $M-M_1$ there
exists a homotopy
\begin{equation}
\theta: \Phi_E \vert_{M - M_1} \sim \Phi_F \vert_{M - M_1}.
\end{equation}
Then there exists a $G$-bundle $\xi_{\theta} \lra S(M_1)$ and orientations $[M]$ and $[S(M_1)]$, such that for any class $P \in H^n(BG,\Lambda)$, we have the following equality of Kronecker pairings as established in \cite[Th. (1.1)]{GGV}:
\begin{equation}\label{pairing-1}
\langle P(E) - P(F), [M] \rangle = \langle P(\xi_{\theta}), [S(M_1)] \rangle .
\end{equation}
\end{theorem}

\begin{remark}
An analogous result in the context of generalized monoidal transformations was given in \cite{GV1}.
\end{remark}

\subsection{The clutching construction}\label{clutching}

We now give a more explicit construction of the $G$-bundle $\xi_{\theta}
\lra S(M_1)$ which can be used for the applications which are to follow.

Suppose now that the $n$-adapted pair $(M,M_1)$ is a pair of closed, connected and oriented manifolds of dimension $n$, respectively of codimension $r \geq 2$ (as in Remark \ref{top-example}; again, this could be taken in the topological, PL, or smooth context). Also, without much loss of
generality, we suppose that $G$ is one of the Lie groups $\SU(\ell), \SO(\ell)$ or
$\Sp(\ell)$, corresponding to a complex, real oriented or symplectic vector bundle of rank $\ell$, respectively. We may also consider $G = {\rm{O}}(\ell)$ and real vector bundles of rank $\ell$ when $\Lambda = \bZ_2$.
For ease of notation, we retain $E$ and
$F$ to denote the associated vector bundles to those $G$-bundles as
previously.

Suppose now we have a homomorphism $\psi: E \lra F$, such that
\begin{itemize}
\item[i)] $\psi:E \vert_{M- M_1} \ovsetl{\cong} F\vert_{M - M_1}$

\item[ii)] $\psi\vert_{M_1}$ has constant rank.
\end{itemize}
Here $\psi$ is viewed as a `clutching function' \cite{Atiyah1,Karoubi} used to clutch $E$ and $F$ over $S(M_1)$ (cf. \cite{Vanque}).
With these assumptions, we have then the following exact sequence of vector bundles on $M_1$:
\begin{equation}
0 \lra K_1 \lra E\vert_{M_1} \ovsetl{\psi} F \vert_{M_1} \lra K_2 \lra 0,
\end{equation}
where $K_1 \cong \ker \psi$,  and $K_2 \cong \coker ~\psi$. Further, let
\begin{equation}\label{ell-def}
L := \psi(E\vert_{M_1}) \subset F\vert_{M_1}.
\end{equation}
From this we may exhibit on $S(M_1)$ vector bundle isomorphisms for the clutched bundle as given by
\begin{equation}
E \vert_{M_1} \cong K_1 \oplus L,~ \text{and}~ F \vert_{M_1} \cong K_2 \oplus L.
\end{equation}
In the development of ideas that follow, the vector bundle $L$ in \eqref{ell-def} along with its characteristic classes
will be essential objects for producing formulas that will specialize the righthand side of \eqref{pairing-1}.

In this setting, a geometric realization of $S(M_1)$ can be given as follows. Let $B(M_1)$ be a closed tubular neighborhood of $M_1$. For $i=1,2$, let $B_i(M_1)$ be two distinct copies of $B(M_1)$.
Identifying $B_1(M_1)$ and $B_2(M_1)$ along their common boundary $\del B(M_1) = \del B_1(M_1) = \del B_2(M_1)$, then $S(M_1)$ may be realized as the `double' $S(M_1)$ obtained by setting
\begin{equation}\label{double-1}
S(M_1) = B_1(M_1) \cup_{\del B(M_1)} B_2(M_1).
\end{equation}
This induces a closed $S^{r}$-fibration $q:S(M_1)
\lra M_1$, with restriction maps
\begin{equation}\label{split-proj}
q_i = q\vert_{B_i(M_1)}: B_i(M_1) \lra M_1,
\end{equation}
that can be seen to be deformation retracts (see e.g. \cite{Atiyah1,Karoubi}).

Observe that $\psi: E \lra F$ is a isomorphism when restricted to the common boundary
$\del B_1(M_1) = \del B_2(M_1)$.

As in \cite{Vanque}, we may exhibit a vector bundle isomorphism
\begin{equation}
\xi= (E, \psi, F) \cong (q_1^*K_1, \eta, q_2^* K_2) \oplus q^*L ,
\end{equation}
where $\eta$ is the appropriate clutching function. For ease of notation, let us set the clutched bundle $(q_1^*K_1, \eta, q_2^* K_2) = K$, so that $\xi$ is expressed as the direct sum
\begin{equation}\label{xi-bundle}
\xi = K \oplus q^*L .
\end{equation}


\subsection{The result in the characteristic ring}

Observe that given an oriented sphere bundle $q: S(M_1) \lra M_1$ with fibre $S^r$ as above, we have the associated (long exact) Gysin sequence (see e.g. \cite{BT,Hirzebruch})
\begin{equation}\label{gysin}
\ldots H^i(S(M_1)) \ovsetl{q_{*}} H^{i-r}(M_1) \ovsetl{\cup e} H^{i+1}(M_1) \ovsetl{q^*} H^{i+1}(S(M_1)) \ldots
\end{equation}
in which the maps $q_{*}$, $\cup e$, and $q^*$ are integration along the fibre, the product with the Euler class, and the natural pull-back, respectively.

For a given vector bundle $E$ of rank $\ell$ over $M$ and with the corresponding characteristic class
$P = P_E \in H^*(BG, \Lambda)$, for $G = \U(\ell), \SO(\ell)$ or $\Sp(\ell)$, typically one considers the image $\Phi^*_E(P)$ in $H^*(M, \Lambda)$.
Recall (from e.g. \cite{BT,Hirzebruch}) that $H^*(BG, \Lambda)$ is a polynomial ring in the Chern classes (for complex vector bundles, with $\Lambda = \bZ$); in the Pontrjagin classes and Euler class (for real oriented vector bundles, with $\Lambda = \bZ[ \frac{1}{2}]$); or the corresponding Pontrjagin classes (for symplectic vector bundles, with $\Lambda = \bZ$), and in the Stiefel-Whitney classes (for real vector bundles, with $\Lambda = \bZ_2$).

A differentiable version of the following result was established \cite{Vanque} and stated explicitly in terms of the Chern forms of a principal ${\rm{U}}(q)$-bundle
(for some $q$).

\begin{theorem}\label{main-2}
Let $K$ and $L$ be as in \eqref{xi-bundle}, with $q: S(M_1) \lra M_1$ the $r$-sphere bundle as above.
With respect to the clutched bundle
$\xi = (E, \psi, F)$ in \eqref{xi-bundle} and the classifying map $\Phi$, suppose 1) there exists a splitting
$\Phi_{\xi}^*(P) = \Phi^*_{K}(P) \cup \Phi^*_{q^* L}(P)$, and 2) the cohomological degree of $P_K \leq \rank(K) = r$.
Then we have the following equality in characteristic numbers
\begin{equation}\label{characteristic-1}
\langle \Phi_{\xi}^*(P), [S(M_1)] \rangle = \langle \Phi^*_{K}(P) \cup \Phi^*_{q^* L}(P), [S(M_1)] \rangle = k \langle \Phi^*_L(P), [M_1] \rangle,
\end{equation}
and hence
\begin{equation}\label{characteristic-2}
\langle P(E)  - P(F), [M] \rangle = \langle P(\xi), [S(M_1)] \rangle = k \langle P(L), [M_1] \rangle,
\end{equation}
for some constant $k \in \Lambda$.
\end{theorem}

\begin{proof}
Let $\a \in \Phi^*_L(P)$, and $\be \in \Phi^*_K(P)$. Then if $q_{*}$ and $q^{*}$ are the maps in \eqref{gysin}, we have via fibre-integration along $S(M_1) \vert_{x \in M_1}$, the equality $q_{*}(q^*(\a) \cup \be) = \a \cup q_{*}(\be)$, which on integrating over $M_1$, yields
\begin{equation}\label{product-1}
\langle \Phi^*_L(P)\langle \Phi^*_K(P), [S(M_1)_x] \rangle, [M_1] \rangle = \langle \Phi^*_K(P), [S(M_1)_x] \rangle \langle \Phi^*_L(P), [M_1] \rangle.
\end{equation}
Now $\langle \Phi^*_K(P), [S(M_1)_x] \rangle = \langle \Phi^*_{K_x}(P), [S(M_1)_x] \rangle$, where $K_x = (q_1^* K_1\vert_{B_1(M_1)_x}, \eta_x,
q_2^* K_2\vert_{B_2(M_1)_x})$ is the vector bundle over $S(M_1)\vert_{x \in M_1}$ constructed via the transition function $\eta_x$ seen as the restriction
of $\eta$ to $\del B(M_1)_x$. That is, we have an isomorphism
\begin{equation}
\eta_x: q^*_1 K_1 \vert_{\del B(M_1)_x} \ovsetl{\cong} q^*_2 K_2
\vert_{\del B(M_1)_x}.
\end{equation}
If $c(x,y)$ is a curve in $M_1$ joining two points $x,y \in M_1$, then the restrictions $K_1\vert_{c(x,y)}$, and $K_2\vert_{c(x,y)}$ are trivial. We have then the following diagram in which the vertical maps are isomorphisms
\begin{equation}
\begin{CD}
q_1^*K_1 \vert_{\del B(M_1)_x}   @> \eta_x >>  q_2^*K_2 \vert_{\del B(M_1)_x} \\
@V \cong VV   @VV \cong V     \\ q_1^*K_1 \vert_{\del B(M_1)_y}      @> \eta_y>> q_2^*K_2 \vert_{\del B(M_1)_y}
\end{CD}
\end{equation}
and modulo these isomorphisms, $\eta_x$, and $\eta_y$ are homotopic. Thus $K_x$ and $K_y$ regarded as bundles on $S^r \cong S(M_1)_x \cong S(M_1)_y$, are isomorphic. Since $M_1$ is connected, this implies $\langle \Phi^*_{K_x}(P),  [S(M_1)_x] \rangle $ is a constant, $k$, say, independent of $x$ from which \eqref{characteristic-1} follows. Then \eqref{characteristic-2} follows by \eqref{pairing-1}.
\end{proof}

In particular, if $\Phi^* P(~)= e(~)$ is the Euler class of a real oriented vector bundle, assumptions 1) and 2) in Theorem \ref{main-2} are satisfied, and we obtain as in \cite{Vanque}:
\begin{corollary}\label{euler-lemma}
For $\xi$ as defined above, we have in terms of Euler classes
\begin{equation}\label{euler-class}
\langle e(E) - e(F), [M] \rangle = \langle e(\xi), [S(M_1)] \rangle = k \langle e (L), [M_1] \rangle,
\end{equation}
where $k \in \bZ$ is a constant.
\end{corollary}

\begin{proof}
It is instructive to include the straightforward proof from \cite[Lemma1.2]{GGV} which implements the maps in \eqref{gysin} (cf. \cite{Vanque}).
Starting from \eqref{xi-bundle}, we have
$$
\begin{aligned}
\langle e(\xi), [S(M_1)] \rangle &= \langle e(q^*L) \cup e(K), [S(M_1)] \rangle \\
&= \langle q^* e(L), e(K) \cap [S(M_1)] \rangle \\
&= \langle e(L), q_{*}(e(K) \cap[S(M_1)]) \rangle \\
&= k \langle e(L), [M_1] \rangle.
\end{aligned}
$$
\end{proof}

\begin{remark}\label{hypothesis-remark}
The applications in the following sections of this paper mainly involve complex vector bundles (so that $\Lambda = \bZ$), with $P$ corresponding to the total (or top) Chern class $c_{*}$ (or $c_{top}$), so that both assumptions 1) and 2) in Theorem \ref{main-2} are satisfied (with $r$ the rank of $L$ as a complex vector bundle), so that $M_1$ is of real codimension $2r$, with cohomological degree $c_{*} (L) \leq 2r$.
\end{remark}


\section{Rational maps of projective varieties}\label{rational-maps-1}

In the following, we shall be applying the general construction and results of \S\ref{topological} in the category of complex manifolds with morphisms the meromorphic maps. Here $M$ and $M_1$ are assumed to be closed connected and oriented smooth manifolds. In this case there will be a slight adjustment in the roles played by $M$ and $M_1$ as result of redefining certain terms.
When the context is clear, it is assumed that complexified tangent bundles are taken in each case.

The natural examples in this context include rational maps of (algebraic) projective varieties, and this fully enriched
situation is the one to which we pay some attention.
But we will point out now (as the astute minded reader can see) that the development of ideas, and constructions, etc.
apply equally well if the spaces in question are just taken to be compact complex manifolds. But restricting to the algebraic
case affords us some access to using significant numerical data, which otherwise might not necessarily be the case in the more general
setting.

\subsection{Application of the general result}\label{application-1}

Let $X$ and $Y$ be compact projective manifolds, $\dim_{\bC}X = \dim_{\bC}Y = n$, and let $f:X \supset U \lra Y$ be a rational map.
In general, such a map will have a closed algebraic \emph{indeterminacy set} $I_f$, namely the locus of points in $X$ for which $f$ fails to be defined,
with open complement $U = X - I_f$.

Let the locally closed complex algebraic subset
\begin{equation}\label{rankset-1}
X(s) :=\{x \in X: \rank_{\bC} (df(x)) \leq s < n \} \subset X,
\end{equation}
be such that
$f^{-1}(f(X(s))= X(s)$.

Letting $Z$ be the complex algebraic subset  $Z = I_f \cup X(n -1)$, we consider $(X,Z)$, or more generally $(M, M \cap Z)$, as an $(n,\bZ)$ adapted pair
(with $r= \codim_{\bC} Z \geq 1$, respectively $2r = \codim_{\bR}(M \cap Z)$), with $M \subset X$ a closed and connected smooth oriented submanifold transversal to $Z$.

More specifically, $M$ does not intersect the singular locus $Z_{sing}$ of $Z$, and is transversal to the regular part $Z_{reg}$ of $Z$, with $M_1 = M \cap Z$ a closed and connected smooth oriented submanifold, $\codim_{\bR} M_1 = 2r \geq 2$. In particular, when $M=X$ is connected of real dimension $2n$, the space $Z_{sing}$ is assumed empty, with $Z$ connected.

\begin{remark}
The point of taking $s = n-1$ in the definition of $Z$, is because at a later stage we will need $\psi = df$ to be an isomorphism outside of $Z$, and $\del B_1(N) = \del B_2(N)$.
\end{remark}

\begin{remark}\label{rational-remark}
We recall from e.g. \cite{GH} that such a rational map $f: X \lra Y$ can be specified by a holomorphic map
$\tilde{f}: X - I_f \lra Y$, for which $\codim_{\bC}I_f \geq2$. Thus for now, we are motivated to take $r \geq 2$, and
view $Z$ as a `singular projective variety'. Also, in cases where $X=Y$, for instance, we might replace $X(s)$ above, by the fixed point set
$\Fix(f)$ of $f$, in the case of a self-map $f:X \lra X$ (see e.g. \S\ref{fixed-point}).
\end{remark}

\subsection{A certain $2p$-cycle and application of Theorem \ref{main-1}}\label{p-cycle}

We fix the smooth submanifold $M \subset X$, with $\dim_{\bR} M = 2p$ (for a fixed $p$ with $1 \leq p \leq n$) that
intersects $Z$ transversally, and set
\begin{equation}\label{N-cap}
N = M \cap Z = M \cap (X(n-1)\cup I_f),
\end{equation}
so that $\dim_{\bC} N = p -r $, for $p \geq r$.
On applying our general result, we note that the isomorphisms in question are simply \emph{topological}
unless otherwise stated.

To this extent we take $(M,N)$ to be a $(2p, \bZ)$-adapted pair, and
take $B(N)$ to be a tubular neighborhood of $N$ in $M$. Thus $N$ now
plays the role of $M_1$ in \S\ref{adapted}. With this slight
modification in mind, we construct as in \S\ref{clutching}, the smooth
double $S(N)$ producing
the $S^{2r}$-fibration $q: S(N) \lra N$, along with projections $q_i: B_i(N) \lra N$ (for $i=1,2$) as in \eqref{split-proj}.
Note this produces a $2p$-cycle $[S(N)]$.

Here we will set
\begin{equation}\label{bundles-1}
E = TX \vert_{B_1(N)}, ~\text{and}~ F = f^* TY \vert_{B_2(N)},
\end{equation}
where, as before,  $\psi = df: TX \lra f^*TY$ is a isomorphism when restricted to
$\del B_1(N) = \del B_2(N)$.

Following Theorem \ref{main-1}, and from the construction of \S\ref{clutching} in the context of complex vector bundles with
structure group ${\rm{U}}(q)$, for some $q$, we straightaway
obtain
\begin{equation}\label{chern-2}
\langle \Phi^*_E(P) - \Phi^*_F(P), [M] \rangle  = \langle \Phi^*_{\xi}(P), [S(N)] \rangle,
\end{equation}
in terms of Chern polynomials $\Phi_{\Diamond}(P)$.

On applying \eqref{N-cap} together with \eqref{bundles-1}, then \eqref{chern-2} with $\Phi_{\Diamond}^*(P) = c_p(\Diamond)$ reduces to the following form (cf. \cite{Brasselet1,Vanque}):
\begin{equation}\label{chern-3}
\langle c_p(X) - f^* c_p(Y), [M] \rangle  = \langle c_p(\xi), [S(N)] \rangle,
\end{equation}
where $\xi = (E, \psi, F)$ is given by \S\ref{clutching}.

\subsection{The righthand side of \eqref{chern-3}}

In order to deal with enumerating the righthand side of \eqref{chern-3}, we return to the setting and conditions of \S\ref{clutching}.
Here we take $\psi$ to have constant rank $n-r$ along $N$, and following \eqref{ell-def} we have the isomorphism
\begin{equation}\label{ell-bundle-1}
TZ \vert_N \cong L = \psi(TZ)\vert_{N},
\end{equation}
so that $\rank_{\bC} L = n-r$.
We also recall from \S\ref{clutching}, the relations
\begin{equation}
\begin{aligned}
E\vert_N &\cong K_1 \oplus L, ~  ~ F\vert_N \cong K_2 \oplus L, \\
E &\cong q_1^*(K_1) \oplus q_1^*(L), ~ ~ F \cong q_2^*(K_2) \oplus q_2^*(L),
\end{aligned}
\end{equation}
while noting that $N$ is a deformation retract of $B_i(N)$, for $i=1,2$.
Hence on $S(N)$ we have the isomorphism $\xi \cong (q_1^* K_1, \eta, q_2^* K_2) \oplus q^*L = K \oplus q^*L $, as in \eqref{xi-bundle}, with
$\rank_{\bC} \xi = n$, from which we deduce $\rank_{\bC} K = r$ (note that we have identified $K_1$ with the restriction to $N$ of a complex rank $r$ vector bundle normal to $TZ$ in $TX\vert Z$).

As deduced from the total Chern classes of $K$ and $q^*(L)$, it is straightforward to show that
\begin{equation}\label{chern-4}
\begin{aligned}
c_k(K \oplus q^*L) &=  \sum_{\nu=1}^r c_{\nu}(K) \cup q^*c_{k - \nu}(L) + q^*c_k(L), ~ ~ 1\leq k \leq n-1, \\
c_n(K \oplus q^*L) &=  \sum_{\nu=1}^r c_{\nu}(K) \cup q^*c_{n - \nu}(L).
\end{aligned}
\end{equation}

\begin{remark}
Following from \cite[\S4]{Brasselet1}(cf. \cite[pp. 408-409]{Schwartz1}), the existence on $S(N)$ of $\ell$ linearly independent trivializing sections
of $q^*(L)$, with $\ell = (n-r) - (k-r) = n-k$, leads to the result
\begin{equation}
\langle q^*c_k(L), [S(N)] \rangle = 0, ~\text{if} ~ 1\leq k \leq n-1.
\end{equation}
\end{remark}

\begin{theorem}\label{main-3}
With regards to \eqref{N-cap} and \ref{ell-bundle-1}, we have
\begin{equation}\label{chern-6}
 \langle c_p(TX\vert M) - c_p(f^*TY\vert M), [M] \rangle  = k \langle c_{p-r} (L), [N] \rangle,
 \end{equation}
 for some constant $k \in \bZ$, with $\dim_{\bR} M = 2p$ and $\dim_{\bR} N = 2(p-r)$.
\end{theorem}

\begin{proof}
In view of Theorem \ref{main-2}, we will here set $E =TX\vert M$, $F= f^*TY\vert M$, and consider the clutched bundle $\xi = (E, \psi, F)$ over $M$. Note that on calling Remark \ref{hypothesis-remark}, the hypotheses 1) and 2) of Theorem \ref{main-2} are satisfied in this case.
To proceed, it suffices to consider total Chern classes $c_{*}$, which are multiplicative, and then commence by substituting this data into the left hand side of \eqref{characteristic-2}.

We have $\rank_{\bC}K =r$, $\rank_{\bC} L = n-r$, and the total Chern class $c_{*}(K)$ is of cohomological degree $\leq 2r = \codim_{\bR} N$, so Theorem \ref{main-2}, in particular \eqref{characteristic-2}, can be applied. Finally, we have $\langle c_{*}(L), [N] \rangle = \langle c_{p-r}(L), [N] \rangle$,
since $\dim_{\bR}N = 2(p-r)$.
\end{proof}

\subsection{Interpreting Theorem \S\ref{main-3}}

In view of applying the clutching construction of \S\ref{clutching}, Theorem \S\ref{main-3} produces a significantly general
formula that can be observed when regulating the topology of a rational map $f:X \lra Y$ of compact
projective varieties of equal (complex) dimension in terms of the cycles $[M]$ and $[N]$ as defined.
Note that $[N]$ is a cycle
which contains part of the (possibly singular) variety $Z=X(n-1) \cup I_f$, once $Z$ is intersected by the $2p$-cycle
$[M]$ as in \eqref{N-cap}.

A working principle is to `resolve'
the indeterminacy set $I_f$, for instance by blowing up along
$f(I_f)$, and then reduce matters to considering a holomorphic map $\hat{f}: X \lra
Y$.

For instance, if $M=X$ (so $p=n$), then \eqref{chern-6} in this case reduces to:
\begin{equation}\label{chern-7}
\langle c_n(X) - \hat{f}^* c_n(Y), [X] \rangle  = k \langle c_{n-r}(X(n-r)), [X(n-r)] \rangle,
\end{equation}
for a constant $k \in \bZ$, with $X(n-r) = \{x \in X: \rank_{\bC} (d\hat{f}(x)) \leq n-r \}$ smooth and connected, $\codim_{\bC} X(n-r) =r \geq 1$, and $d\hat{f}$ of constant (complex) rank $n-r$ along $Z=N=X(n-1)=X(n-r)$ (so that $L \cong TZ \cong TX(n-r)$).

More specifically, given $f: X \lra Y$, one may pass to a
proper modification $\hat{f} : \what{X} \lra Y$, and apply the formula
in Theorem \ref{main-3} for a holomorphic map, provided
  $\what{X}$ is a compact complex manifold.
  \footnote{We recall from \cite{Grauert} that  a \emph{proper modification} $\hat{f} : \what{X} \lra Y$
  means that $\hat{f}$ is a proper surjective holomorphic map, such that
  there exists nowhere dense analytic subsets $X' \subset X$ and $Y'
  \subset Y$, such that: i) $\hat{f}(X') \subset Y'$, ii) $\hat{f}: \what{X} - X' \lra
  Y - Y'$ is a biholomorphism, and iii) each fiber $\hat{f}^{-1}(y)$, for $y
  \in Y'$, consists of more than one point. The set $X'=
  \hat{f}^{-1}(y)$ is called the \emph{exceptional set}.
Note that $\what{X}$ may not necessarily be an algebraic variety in general, even if $X$
has this property \cite{Hironaka}.}

\subsection{The case of a holomorphic ramified covering map}\label{holo-ramified}

Consider the case where $s=n-1$ is constant, and let $f: X \lra Y$ be a holomorphic branched covering map
for which $X_1 := X(n-1) \subset X$ is the ramification divisor on which $\rank_{\bC}
f\vert_{X_1} = n-1$, with $r= \codim_{\bC} X_1= 1$. As before, let $M \subset X$ be any compact oriented
smooth submanifold with $\dim_{\bR} M = 2p$ that meets $X_1$
transversally (with $1 \leq p \leq n$), with  $N = M \cap X_1$. From \eqref{ell-bundle-1}, we have $L = TX_1 \vert _N$. This leads to a version of the higher dimensional Riemann-Hurwitz formula as given in \cite[Proposition 2]{Brasselet1} (see also \cite[p. 409]{Schwartz1}):
\begin{equation}\label{chern-8}
\langle f^*c_p(Y) - c_p(X), [M] \rangle = (\mu - 1)
\langle c_{p-1}(X_1), [N] \rangle,
\end{equation}
where $\mu = \deg(f \vert_{X^{(1)}}) \in \bZ$ is the local topological degree of $f$ along $X_1$.
Note that in general, the global degree $\deg(f) := \delta \neq \mu$.

Together with \eqref{chern-6}, this case reveals the interest in enumerating the righthand side of
\eqref{pairing-1} and \eqref{characteristic-1} in general.

\begin{example}
Let $n=2$, where $X$ is now a compact complex surface, and let $Y= \bC P^2$. Consider a holomorphic map $f: X \lra \bC P^2$ which is branched over a curve $C$ of genus $g$ with normal crossings. Thus $f: X - X_1 \lra \bC P^2 - C$ is an unramified covering map of degree $\deg f = \delta$ say, where
the branch set  $X_1 = f^{-1}(C)$.

As an example of showing consistency in the data, suppose in this case $X$ is a K-3 surface. One can construct a map with $\delta = \mu =4$, branched over a quartic curve $C$ (see e.g. \cite{Morrison}). In this top dimension, we have
$\langle c_2(X), [X] \rangle = \chi(X)= 24$, and $\langle f^*c_2(\bC P^2), [X] \rangle = 4(\chi(\bC P^2) = 4(3) = 12$. From \eqref{chern-8}, with $n=p=2$, and
the adjunction formula
(e.g. \cite[p.221]{GH}), it is straightforward to see that $\langle c_1(X_1), [X_1] \rangle = \chi(C)= -4$, and therefore $g=3$.
\end{example}


\section{Rational self-maps}\label{rational-maps-2}

\subsection{Iterated self-maps of projective varieties}

Now we consider the case where $X=Y$, and $f: X \lra X$ is a rational map with $k$-th iterate denoted $f^k$.
Let
\begin{equation}\label{rankset-2}
X(k,s):=\{x \in X: \rank_{\bC} (df^k(x)) \leq s < n \},
\end{equation}
and set $Z = X(k,n-1)\cup I_{f^k}$, with $\codim_{\bC} Z = r_k$.
Let $\Delta_k = \deg(f^{k})$, and note that in general, $\Delta_k \neq \delta^k = \deg(f)^k$.

We recall from \S\ref{application-1}, making slight modifications, that $M \subset X$ is a closed and connected smooth oriented submanifold, $\dim_{\bR}M =2p$, transversal to $Z = X(k,n-1)\cup I_{f^k}$. Again, $M$ does not meet $Z_{sing}$, and is transversal to $Z_{reg}$, with $M_1 = M \cap Z$ a closed and connected smooth oriented submanifold, $\codim_{\bR} M_1 = 2r_k \geq 2$. When $M =X$ is connected, $\dim_{\bR}M =2n$, then $Z_{sing} = \emptyset$, and $Z$ is connected.

Also, we take $E = TX \vert_{B_1(N)}$, and $F = (f^k)^* TX \vert_{B_2(N)}$.
These are only the essential differences,
otherwise the basic construction leading to the various formulas remains the same. In particular, Theorem \ref{main-3} holds with $X(n-1)$ replaced by $X(k,n-1)$ in \eqref{rankset-2}.

\subsection{Holomorphic and analytically stable maps}

Let us deal first with a holomorphic map $f:X \lra X$, where the $k$-th iterate $f^k$ is a
holomorphic branched covering map with ramification divisor $X_{1,k} := X(k, n-1)$, with $r_k = 1$. Then \eqref{chern-8} reads
as
\begin{equation}\label{chern-9}
\langle (f^k)^*c_p(X) - c_p(X), [M] \rangle = (\mu_k - 1)
\langle c_{p-1}(X_{1,k}), [N] \rangle,
\end{equation}
where $\mu_k \in \bZ$ is the local topological degree of $f^k$ along $X_{1,k}$.
We have $\Delta_k \neq \mu_k$, in general.

However, in this case where $f$ is holomorphic and $X= \bC P^n$, we do have $\Delta_k = \delta^k$ \cite{Fornaess1}.
On the other hand, it is clear for $M = \bC P^n$, and $p=n$, that such maps are thus regulated by the expression derived from \eqref{chern-9}:
\begin{equation}\label{chern-10}
(n +1)(\Delta_k - 1) = (\mu_k - 1)
\langle c_{n-1}(X_{1,k}), [X_{1,k}] \rangle.
\end{equation}
For instance:
\begin{proposition}
Let $f: \bC P^2 \lra \bC P^2$ be a holomorphic map. Suppose that the $k$-th iterate $f^k: \bC P^2 \lra \bC P^2$ is a holomorphic branched covering map over a curve of genus $g$. Then necessarily the global degree $\Delta_k \neq \mu_k$ (the local degree).
\end{proposition}
\begin{proof}
A straightforward enumeration of \eqref{chern-10} leads to
\begin{equation}
3 (\Delta_k - 1) = 3 (\delta^k - 1) = (\mu_k - 1) (2- 2g),
\end{equation}
which is meaningless if $\Delta_k = \mu_k$.
\end{proof}

Taking $n=2$, a bimeromorphic map $f: X \lra X$ is said to be \emph{analytically stable} if:  i) for all $k \geq 0$, we have
$(f^k)^* = (f^*)^k$, and ii) for each curve $C$ in $X$, $f^k(C) \notin I_f$ (see \cite{Fornaess1}).
Recall that the blow-up of $Y$ at $y$,
is the proper modification $\hat{f} : \what{X} \lra Y$ which
replaces $y$ with the exceptional curve $\pi^{-1}(y) \cong \bC P^{1}$,
the set of holomorphic tangent directions at $y$, and $\hat{f}$ is a
biholomorphism elsewhere.
In fact, in this instance, any proper modification $\hat{f} : \what{X}
\lra Y$
arises as a composition of finitely many point blow ups
(see e.g. \cite[Th. 1.1]{Diller}).
Following \cite[Th. 01]{Diller},
if $f: X \lra X$ is a bimeromorphic map, then there always
exists a proper modification
$\hat{f} : \what{X} \lra X$ that lifts $f$ to an analytically stable map.

\begin{remark}
In \cite{Bonifant1} it is shown, that for $n=2$, there are countably many sequences $\{d_{\ell}\} \subset \bN$ for which
a rational map $f: \bC P^2 \lra \bC P^2$ exists, satisfying $\Delta_{\ell} = d_{\ell}$, for all $\ell \in \bN$.
\end{remark}
It would be interesting to see how the topology of these procedures can be regulated by results of the type Theorem \ref{main-3}, and by the known results for the Chern classes of the blowing-up process (e.g. \cite{Aluffi,Porteous1}).

\subsection{Degree lowering curves}

For a general rational map $f:\bC P^n  \lra \bC P^n $, there may be any amount of peculiar behavior. For instance, it is possible that an iteration
$f^k$ may, for some $k$, map an (irreducible) curve $C$ into the indeterminacy set $I_{f^k}$ (thus $f$ cannot be analytically stable). In this case one sees that $\Delta_k < \delta^k$ \cite{Fornaess1}, and so enumerating \eqref{chern-6}, for $p=n$, can easily be seen to give
\begin{equation}
\Delta_k = 1 - \frac{1}{n+1}(k \langle c_{n-r}(L), [N] \rangle) < \delta^k,
\end{equation}
where $N$ is given by \eqref{N-cap}, $\codim_{\bR} N = 2r$, and $L$ is as in \eqref{ell-bundle-1}.

\subsection{Holomorphic maps with fixed point set}\label{fixed-point}

Next we will consider a holomorphic map $f: X \lra X$, with a (possibly singular) fixed point set $S = \Fix(f)$, with $r= \codim_{\bC} S$.
In the general setting of \S\ref{topological}, we will regard $(X,S)$ or $(M,M\cap S)$ as an $(n, \bZ)$-adapted pair with $S$ playing the role of $Z$
as in \S\ref{p-cycle}, and $M_1 = N = M \cap S$. Again, the transversality and other conditions relating to $M$, $Z_{reg}$, and $Z_{sing}$ as made in \S\ref{application-1} and \S\ref{rational-maps-2}, also apply here.

This application is partially motivated by
certain constructions in  \cite{Abate1} (where $X$ can be taken more generally to be a complex manifold) to which we will apply the general results
of Theorem \ref{main-1} and Theorem \ref{main-2}. But this will necessitate assigning some different data compared to that of the previous sections and re-defining terms accordingly.

Specifically, let us start by letting $Q_S$ denote the normal bundle to $S$ in $X$. Then let
\begin{equation}\label{bundles-2}
E = (Q_S)^{\otimes \nu_f},~ \text{and}~ F = TX,
\end{equation}
where $\nu_f \in \bN$. We take up the hypotheses in \S\ref{clutching} in terms of a homomorphism $\psi: (Q_S)^{\otimes \nu_f} \lra TX$. In particular, as in \cite{Abate1}, the restricted morphism $\psi\vert_{S}: (Q_S)^{\otimes \nu_f} \lra TX\vert_{S}$, is identified with a holomorphic section
of $TX \vert_{S} \otimes (Q_S^*)^{\otimes \nu_f}$. The bundles $K$ and $L$ are taken to be as in \S\ref{clutching}, with
\begin{equation}\label{ell-bundle-2}
L = \IM(\psi \vert_{S}) \subset TX \vert_S.
\end{equation}
Let $M$ be a $2p$-dimensional submanifold $M \subseteq X$ intersecting $S$ transversally. As before, we set $N = M \cap S$, with
$\dim_{\bC} N = p -r$ (for $p \geq r$).

Having made these adjustments for the bundles $E$ and $F$, etc., we apply Theorem \ref{main-1} together with essentially the same proof as that used in proving Theorem \ref{main-3}, to obtain
\begin{equation}\label{chern-11}
\langle c_p((Q_S)^{\otimes \nu_f}) - c_p(X), [M] \rangle = k
\langle c_{p-r}(L), [N] \rangle.
\end{equation}

\begin{example}
Consider the case $r=1$, and $S$ is a (possibly singular) hypersurface regarded as an oriented $(n-1)$-circuit. Then we have $\rank_{\bC}Q_S =1$, and $\rank_{\bC} L = n-1$, following which \eqref{chern-11} gives
\begin{equation}\label{chern-12}
\langle c_p((Q_S)^{\otimes \nu_f}) - c_p(X), [M] \rangle = k
\langle c_{p-1}(L), [N] \rangle,
\end{equation}
with $c_1((Q_S)^{\otimes \nu_f}) = \nu_f c_1(\mathcal{O}_S)$, and $c_p((Q_S)^{\otimes \nu_f}) =0$ for $p\geq 2$ (since $Q_S$ is a line bundle).
Note that applying \eqref{chern-12} for $M=X$ ($p=n$), we have $N=S$.
\end{example}

\begin{remark}
In the setting of \cite{Abate1}, the quantity $\nu_f$ is considered as a measure of `order of contact' between the map $f$ and $S$.
For the case of such a hypersurface $S$ there are the connected components $\Lambda_{\a}$ of the union of singular sets $\Sing(\fX_f) \cup \Sing(S)$, where $\Sing(\fX_f)$ is the set of zeros of a vector field $\fX_f$ associated to $f$ that induces a (generally singular) holomorphic foliation.
This leads to a residue formula $\sum_{\a} \Res(\fX_f, S, \Lambda_{\a}) = \langle c_1^{n-1}(S), [S] \rangle$ as in \cite[Th. 01]{Abate1}.
Our approach leads to somewhat different formulas as seen above. Though enumerating \eqref{chern-12} for the case $n=p=2$, reveals the right-hand (up to a constant) to be such a residual quantity.
\end{remark}


\section{Determinantal varieties}

\subsection{The degeneracy locus}\label{degeneracy-locus}

In this section we commence the second part of the
paper by turning to a related, but essentially more general setting.
In the previous sections we considered applying the general
result of \S\ref{topological} to rational maps of projective
varieties. But now we tweak the setting of those sections somewhat
with several terms redefined for the sake of replacing maps of
projective varieties by vector bundle morphisms over a compact complex
manifold $X$.

More specifically, consider a morphism $\psi: E \lra F$ of complex
(smooth) vector bundles of the same complex rank $\ell \geq 1$, over a projective variety $X$ (where $\dim_{\bC} X = n$). For some
given $s \in \bN$, we have the \emph{degeneracy locus} of $\psi$, as
defined by
\begin{equation}\label{degeneracy-1}
\Omega(s) := \{ x\in X: \rank_{\bC} \psi(x) \leq s < \ell\}.
\end{equation}
As in \S\ref{p-cycle}, we fix a smooth submanifold $M \subset X$, with $\dim_{\bR} M = 2p$ ($1 \leq p \leq n$) that intersects $\Omega(\ell-1)$ transversally (as previously $s=\ell-1$ is taken since we require $\psi$ to be an isomorphism outside of $\Omega(\ell-1)$).

Once again we will apply the general setting of \S\ref{topological},
where $(X, \Omega(\ell-1))$ (respectively, $(M, M \cap \Omega(\ell-1))$) is regarded as an $n$-adapted pair, so that $\Omega(\ell-1)$ (respectively, $M \cap \Omega(\ell-1)$) plays the role of $M_1$ in \S\ref{topological}), with $r=
\codim_{\bC} \Omega(\ell-1) \geq 1$.

Theorem \ref{main-1}, in particular, \eqref{pairing-1} immediately
applies to give the general statement
\begin{equation}\label{degeneracy-2}
\langle P(E) - P(F), [X] \rangle = \langle P(\xi_{\theta}),
[S(\Omega(\ell-1))] \rangle.
\end{equation}
\begin{remark}
The cohomology class $\{ \Omega(s) \}$ of $\Omega(s)$ in $X$ can be
determined by polynomials in the Chern classes of $E$ and $F$, and in
certain cases the codimension of $\Omega(s)$ can be determined (see
\cite{Fulton1,HT} which also cover a historical background to the general problem in the algebraic geometric
context). Note that \cite{HT} deals initially with results in the differentiable category, thus
\eqref{degeneracy-2} applies in that case as well.
\end{remark}

\begin{example}\label{d-example-1}
If $E \lra X$ is a complex vector bundle, $\rank_{\bC} E = \ell$, $F =
E^*$, and $\psi: E \lra E^*$ is a general symmetric bundle map, then following \cite[Th. 1]{HT}, the cohomology class
 $\{ \Omega(s)\}$ is given by a polynomial $P_s(c_1(E^*), \dots,
 c_{\ell}(E^*))$.
\footnote{Let $V$ be a vector space and $V^*$ its dual vector space. A linear map $\psi: V \lra V^*$ is said to be
 \emph{symmetric} if $(\psi(x), y) = (\psi(y), x)$ , for all $x,y \in V$, where $(~,~)$ is the dual pairing between
 $V^*$ and $V$. Equivalently, $\psi$ is symmetric if $\psi = \psi^T$. The precise meaning of `general' is explained in
 \cite[Note 2, p.72]{HT}. Likewise, $\psi$ is \emph{skew-symmetric} if $\psi = - \psi^T$.}
This latter polynomial has an explicit expression given in terms of the determinant
\begin{equation}\label{det-1}
2^{\ell-s}~
\begin{bmatrix}
c_{\ell-s} & c_{\ell-s+1} & c_{\ell-s+2}  &  \\
c_{\ell-s-2} & c_{\ell-s-1} & c_{\ell-s} & \\
\vdots &  & c_{\ell -s -2}  &  \\
\vdots & \dots & &\ddots c_1
\end{bmatrix} = \{\Omega(s)\},
\end{equation}
where $c_i= c_i(E^*)$, and further, $\Omega(s)$ has codimension $r = \binom{\ell-s+1}{2}$ (for $s < \ell$).
\end{example}

In keeping with the previous sections, we will be interested in finding
expressions linking the Chern classes of some degree ($p$, say) of the
spaces in question.

Let $N = M \cap \Omega(\ell-1)$, with
$\dim_{\bC} N = p -r$, for $p \geq r$ (again, the case $r \geq p$ can be treated likewise).
We recall the tubular neighborhoods
$B_i(N)$ of $N$ (for $i=1,2$), and consider $E,F$ as restricted to
$B_1(N)$ and $B_2(N)$ respectively, so that in accordance with
 \S\ref{clutching}, we have $E\vert \del B_1(N) \cong F \vert
 \del B_2(N)$.
Note that we do not yet assume that $\psi$ has constant rank on
$\Omega(\ell-1)$, since we are still in the context of Theorem \ref{main-1}.
We also recall from \S\ref{clutching} the $S^{2r}$-fibration $q: S(N) \lra N$.

We shall be applying the same basic strategy as in
\S\ref{rational-maps-1} (and in \S\ref{rational-maps-2}). Thus \eqref{degeneracy-2}
reduces to a statement that is more general than \eqref{chern-3}:
\begin{equation}\label{d-chern-1}
\langle c_p(E) - c_p(F), [M] \rangle  = \langle
P(\xi_{\theta})^{[2p]}, [S(N)] \rangle,
\end{equation}
where $P(\xi_{\theta}) \in H^*(BU(\ell))$ is a characteristic class whose component $P(\xi_{\theta})^{[2p]} \in H^{2p}(BU(\ell)$ corresponds to the $p$-th Chern class $c_p$ (e.g. $P$ corresponds as before to the total Chern class $c_{*}$).
Again, it is interesting to enumerate the right-hand side
of \eqref{d-chern-1} once the cohomology class $\{\Omega(\ell-1)\}$ has been determined.

\begin{remark}
In view of this last comment, if the cohomology class $\{ N\}$ of $N$ in $M$ happens to be cohomologous to
$\{ \Omega(\ell-1) \}^{[2p]}$ in $H^*(\Omega(\ell-1), \bZ)$, then the class $P(\xi_{\theta})^{[2p]}$ is expressible in
the form
\begin{equation}\label{gamma-form}
P(\xi_{\theta})^{[2p]} = q^* \{ \Omega(\ell-1) \}^{[2p]} \cup \gamma,
\end{equation}
for some $\gamma \in H^{2(p-r)}(S(N), \bZ)$.
\end{remark}
The following observations summarized as a proposition shows that, in the context of the symmetric bundle map of Example \ref{d-example-1}, there is indeed a restriction on
components of the class $P(\xi_{\theta})$.

\begin{proposition}\label{d-prop-1}
With regards to the context of Example \ref{d-example-1},
we have the following relationships:
\begin{itemize}
\item[(1)] For $p$ odd,
\begin{equation}\label{d-chern-2}
2\langle c_p(E), [M] \rangle =  \langle
 P(\xi_{\theta})^{[2p]}, [S(N)] \rangle.
 \end{equation}

\item[(2)]
For $p$ even, the class $P(\xi_{\theta})^{[2p]}$ is trivial
in $H^{2p}(S(N), \bZ)$.
\end{itemize}
\end{proposition}
\begin{proof}
Noting that $c_p(E^*)
= (-1)^p c_p(E)$ (see e.g. \cite[p.411]{GH}), we have from
\eqref{d-chern-1}
\begin{equation}\label{d-chern-3}
(1 + (-1)^{p+1}) \langle c_p(E), [M] \rangle = \langle
 P(\xi_{\theta})^{[2p]}, [S(N)] \rangle,
\end{equation}
for $1 \leq p \leq \ell$, from which the results follow.
\end{proof}
Further enumeration of the righthand side of \eqref{d-chern-3}, can be carried out under the conditions of \S\ref{clutching}
which we will deal with next.

\subsection{Constant rank case}
\label{constant-1}

Suppose now that $\rank_{\bC} \psi
\vert_{\Omega(\ell-1)} = s < \ell$, is constant (so that $\Omega(s) = \Omega(\ell-1)$), and the bundle map $\psi$ is taken
as a clutching map between $E$ and $F$, as in \S\ref{clutching}. In this case, there is the class of the $2p$-component given by
$P(\xi_{\theta})^{[2p]} = c_p(E, \psi, F)$.
Also, Theorem \S\ref{main-2}
applies to give
\begin{equation}\label{d-chern-4}
\langle P(E) - P(F), [X] \rangle = k\langle P(L),
\Omega(\ell-1)] \rangle,
\end{equation}
where, as before $L = \psi(E\vert_{\Omega(\ell-1)}) \subset F\vert_{\Omega(\ell-1)}$, with $\rank_{\bC} L =\ell-1$. In particular,
$P(L) \in H^*(\Omega(\ell-1), \bZ)$; so knowing the cohomology of $\Omega(\ell-1)$ gives us a handle on the class $P(L)$.
Applying Theorem
\ref{main-2} (cf. Theorem \ref{main-3}), we then have
\begin{equation}\label{d-chern-5}
\langle c_p(E) - c_p(F), [M] \rangle  = k\langle
c_{p-1} (L), [N] \rangle.
\end{equation}

\begin{example}
In view of the above remarks, let us return to the context of Example \ref{d-example-1}.
Here we have $L = \psi(E\vert_{\Omega(\ell-1)}) \subset E^*\vert_{\Omega(\ell-1)}$, and the cohomology class $P(L) = P(c_1(E^*), \dots, c_{\ell-1}(E^*))$.
On applying \eqref{d-chern-5}, we thus obtain for
$p$ odd,
\begin{equation}\label{d-chern-6}
2\langle c_p(E), [M] \rangle = k \langle c_{p-1}(E^*), [N] \rangle =
 \kappa \langle  P(c_1(E^*), \dots, c_{\ell-1}(E^*)), [N] \rangle,
\end{equation}
for some constants $k$ and $\kappa$.
\end{example}

\begin{example}\label{hodge}
There are analogous results in \cite{HT} for the cohomology class $\{\Omega(s) \}$ in the case of skew-symmetric maps (morphisms)
$\psi: E \lra E^*$. The cases $\psi: E \lra E^* \otimes \mathcal{L}$, for $\mathcal{L}$ a complex line bundle, are also studied in the symmetric and skew-symmetric cases. For instance, when $\psi: E \lra E^* \otimes \mathcal{L}$ is a general symmetric bundle map, then the cohomology class
$\{\Omega(s)\}$ is given by \eqref{det-1}, but now taking $c_i = c_i(E^* \otimes \sqrt{\mathcal{L}})$ \cite[Th. 10]{HT}.
The (general) skew -symmetric case can likewise be treated.
\end{example}

\begin{example}
(Application to variation of Hodge structure following \cite{HT}):
Consider a family $\varpi: \mathcal{C} \lra X$ of curves of genus
$g$. Let $\sfH^{1,0}, \sfH^{0,1}$ denote the corresponding Hodge
bundles. We have then a period map $\Upsilon: X \lra
\Gr(g,2g)/\Gamma$, where $\Gr(g,2g)$ denotes a certain isotropic
Grassmannian, and $\Gamma \subset \Aut(\Gr(g, 2g))$ is a discrete
subgroup. Consequently, there is a bundle morphism $\psi: TX \lra
\Hom(\sfH^{1,0}, \sfH^{0,1})$, that can be expressed alternatively as a symmetric
bundle map
$TX \lra S^2(\sfH^{1,0})^*$ \cite{HT}.

When $X$ is an algebraic curve of genus $g_X$, and there are no
singular fibres of $\varpi$, then one can enumerate matters as follows.
Setting $E = \sfH^{1,0}$ (so $\rank_{\bC} E = g$), we have
from \cite[p.82]{HT}
$c_1(\det(S^2E^* \otimes \mathcal{O}(1)) \geq 0$. Observing that
$c_1(S^2E^*) = (g+1)c_1(E^*)$, then this previous expression
simplifies to $g(g_X - 1) \geq c_1(E)$.

It can be argued that if the variation of Hodge structure over $X$ is
non-trivial, then by the local Torelli theorem, the period map
$\Upsilon$ has maximal rank at some point of $X$, and by
\cite[Th. 10]{HT}, the degeneracy locus $\Omega(g-1)$ in this case, is
not all of $X$. In the context of a general symmetric bundle map, here given by $\psi: E
\lra E^* \otimes \mathcal{O}(-1)$, and from the remarks in Example
\ref{hodge} above, it follows that the cohomology class
$\{\Omega(g-1)\} = -2c_1(E \otimes \mathcal{O}(-\frac{1}{2})) = -2
(c_1(E) - g(g_X -1))$. This provides us with an enumeration of
\eqref{d-chern-5} with $M=X$,~$N=\Omega(g-1)$, and $F =
\mathcal{O}(-\frac{1}{2})$, in the case of $p=1$ and $s = g-1$. In this case there is just a single
constant $k=k_{\nu} = -\frac{1}{2}$.
\end{example}

\subsection{Flagged bundles}\label{flagged}

Suppose now we consider, as in \cite{Fulton1}, the more general situation of \S\ref{degeneracy-1} for a morphism $\psi: E \lra F$, over $X$,
for which
\begin{equation}\label{flag-1}
\begin{aligned}
E_1 &\subset E_2 \subset \cdots \subset E_u = E\\
F&= F_v \thra F_{v-1} \thra \cdots \thra F_1
\end{aligned}
\end{equation}
are flags of subbundles and quotient bundles, respectively. Here we will take integers $s(\a,\be) \in \bN$ specified across the
intervals $1 \leq \a \leq u$, and $1 \leq \be \leq v$, and the degeneracy locus is then defined by
\begin{equation}\label{flag-2}
\Omega(\mathbf{s}) := \{ x \in X : \rank_{\bC} (E_{\a}(x) \lra F_{\be}(x)) \leq s(\a,\be), \forall \a, \be \},
\end{equation}
where $\mathbf{s}$ is regarded as a certain rank function.
As shown in \cite{Fulton1}, conditions on $\mathbf{s}$ determine the irreducibility of $\Omega(s)$
as a projective variety, and further, the cohomology class
$\{\Omega(\mathbf{s})\}$ can be determined in terms of the Chern
classes of $E$ and $F$.

In the case $\psi \vert_{\Omega(\mathbf{s})}$ has constant rank
$s(\a, \be)$, and $\psi$ is a clutching map as before, we apply
Theorem \S\ref{main-2}
to obtain
$\langle P(E) - P(F), [X] \rangle = k\langle P(L),
[\Omega(\mathbf{s})] \rangle$.

\subsection{Complete flags}
\label{complete}

Following \cite{Fulton1}, we will exemplify matters in
the case of `complete flags' for the data $u=v=m$, and $E_i, F_i$ having (complex) rank $i$. In this case, $\Omega(\mathbf{s})$ is characterized by permutations in the symmetric group $S_m$. Given $w \in S_m$, let $\ell(w)$ be the length of $w$ (in other words, the number of inversions). Let $\mathbf{s}_w(\be,\a) = \text{card}\{ i \leq \be : w(i) \leq \a \}$, and
\begin{equation}\label{flag-3}
x_i = c_1(\Ker(F_i \thra F_{i-1})), \text{and}~ y_i = c_1 (E_i/E_{i-1}), ~\text{for}~ 1 \leq i \leq m.
\end{equation}
Then one restricts attention to
\begin{equation}\label{flag-4}
\Omega(w) = \Omega(\mathbf{s}_w) := \{ x \in X: \rank_{\bC} (E_{\a}(x) \lra F_{\be}(x)) \leq \mathbf{s}_w(\be, \a), \forall \a, \be \}.
 \end{equation}
 Here we make several observations from \cite{Fulton1}:
 \begin{itemize}
 \item[(i)]
 The space $\Omega(w)$ has a natural structure of a scheme given by the vanishing of
 the induced maps from $\bigwedge^{\mathbf{s}_w(\be, \a) +1}(E_{\a}) \lra \bigwedge^{\mathbf{s}_w(\be, \a) +1}(F_{\be})$.

 \item[(ii)]
 The expected (maximum) value of $r = \codim_{\bC} \Omega(w)$, is
 $r= \ell(w)$.

 \item[(iii)]
 The cohomology class $\{\Omega(w)\} = \mathfrak{S}_w(x,y)$, where
 $\mathfrak{S}_w(x,y) = \mathfrak{S}(x_1, \ldots, x_m, y_1, \ldots, y_m)$, is the double Schubert polynomial for $w$, this being a
 homogeneous polynomial in the $2m$ variables of degree $\ell(w)$ (see \cite{Fulton1} for details of the latter).
 \end{itemize}
Theorem \ref{main-1} applies directly to give
\begin{equation}\label{flag-5}
\langle P(E) - P(F), [X] \rangle = \langle P(\xi_{\theta}),
[S(\Omega(\mathbf{s}))] \rangle,
\end{equation}
In the case $\psi\vert_{\Omega(w)}$ has constant rank (less than maximal), we deduce from Theorem \ref{main-2}, that
\begin{equation}\label{flag-6}
\langle P(E) - P(F), [X] \rangle = \langle \what{\mathfrak{S}}_w(x,y),
[\Omega(\mathbf{s})] \rangle,
\end{equation}
where $\what{\mathfrak{S}}_w(x,y)$ is a double Schubert polynomial in
the class $\{\Omega(w)\}$. Thus, with respect to the cycles $[M]$ and $[N]$ as previously defined, we have
\begin{equation}\label{flag-7}
\langle c_p(E) - c_p(F), [M] \rangle = \langle( \what{\mathfrak{S}}_w(x,y))^{[2p]},
[N] \rangle.
\end{equation}

\subsection{Final remark and a further example}\label{final-rem}

We have already mentioned, in the Introduction, the modification of the main result of \cite{GGV} to the topology of generalized monoidal transformations \cite{GV1,GV2}. In closing, we should add that there are likely to be further situations to which Theorem \S\ref{main-1} can be applied. As an example of such a situation, in a similar vein to the development of \S\ref{degeneracy-locus}, consider the following.
\begin{example}\label{todd}
This follows from \cite{Porteous2}. Let $\mathcal{L} \lra X$ be complex line bundle ($X$ here can be a complex manifold), and let
$h: \mathcal{L} \lra \bC^{\ell+1}$ be a transversal linear system on $X$ in the sense of \cite{Porteous2}. Let $E= Q\mathcal{L}$ be the vector bundle on $X$ whose sections consist of the $\bC$-invariant vector fields on $\mathcal{L}$, and let $F= \Hom(Q\mathcal{L}, \bC^{\ell+1})$.
From $h$, one can define a complex vector bundle morphism $\psi: E \lra F$, whose singular set, called the Jacobian set $J(h)$, can be formulated
in a similar way to \eqref{degeneracy-1} (and plays a similar role to $\Omega(s)$). The main results of \S\ref{topological} likewise apply to the adapted pair $(X, J(h))$ (cf. \eqref{degeneracy-2}), and further enumeration in the constant rank case produces a formula similar to \eqref{d-chern-5}
for Chern classes of appropriate order (cf. \cite{Porteous2}).
\end{example}

\bigbreak
\noindent
\textbf{Acknowledgement}
It has been a pleasure for us to dedicate this paper in recognition of Professor Seade's remarkable contributions to research up to his 60-th year, and hopefully beyond as well. Our contribution to the Proceedings benefitted enormously from a substantial and painstakingly detailed report from an anonymous referee who pointed out several corrections and suggested significant improvements in the general presentation. Thus we express our sincere gratitude to the referee for providing such an excellent report towards revising an earlier version of this paper.

\medn
A. Verjovsky was financed by grant IN103914, PAPIIT, DGAPA, of the
Universidad Nacional Aut\'onoma de M\'exico.


James F. Glazebrook.\\
 Department of Mathematics and Computer
Science \\
 Eastern Illinois University \\
600  Lincoln Ave., Charleston, IL 61920--3099 USA \\
jfglazebrook@eiu.edu
\\ (Adjunct Faculty)
\\ Department of Mathematics \\ University of Illinois at
Urbana--Champaign\\ Urbana, IL 61801, USA\\

Alberto Verjovsky\\Instituo de Matem\'{a}ticas\\
Universidad Aut\'{o}noma de M\'{e}xico\\
Av. Universidad s/n, Lomas de Chamilpa\\
Cuernavaca CP 62210, Morelos, Mexico\\
alberto@matcuer.unam.mx

\end{document}